\documentclass{amsart}

\usepackage{amsthm,amssymb,latexsym,amsmath,amsfonts,amscd,amstext}
\usepackage[all]{xy}
\usepackage[english]{babel}
\usepackage[utf8x]{inputenc} \usepackage[T1]{fontenc}
\usepackage{color}
\usepackage{graphicx}	
\usepackage{mathrsfs}
\usepackage[all]{xy}
\usepackage{yfonts}
\usepackage{latexsym}
\usepackage{epsfig}
\usepackage[mathscr]{eucal}

\usepackage{tikz}
\usepackage{tikz-cd}
\usetikzlibrary{arrows,automata}
\usetikzlibrary{fit}
\usepackage{wrapfig}

\newcommand{\p}[1]{{\mathbb{P}^{#1}}}
\newcommand{\calf}{{\mathcal F}}

\newcommand{\GL}{\operatorname{GL}}
\newcommand{\rk}{\operatorname{rk}}

\DeclareMathOperator{\im}{{im}}
\DeclareMathOperator{\Hom}{Hom}
\DeclareMathOperator{\End}{End}

\newtheorem{thm}{Theorem}

\newtheorem{lem}[thm]{Lemma}
\newtheorem{pps}[thm]{Proposition}
\newtheorem{dfn}[thm]{Definition}
\newtheorem{obs}[thm]{Remark}

\title[Obstruction theory for moduli spaces of framed flags]{Obstruction theory for moduli spaces of framed flags of sheaves on the projective plane}

\author{Rodrigo A. von Flach}
\address{Centro de Ci\^encias Exatas e Tecnol\'ogicas \\
Universidade Federal do Rec\^oncavo da Bahia \\
Rua Rui Barbosa, 710 \\
44380-000 Cruz das Almas, BA, Brazil}
\email{rodrigovonflach@ufrb.edu.br}

\author{Marcos Jardim}
\address{IMECC - UNICAMP \\
Departamento de Matem\'atica \\ Rua S\'ergio  Buarque de Holanda, 651\\
13083-970 Campinas-SP, Brazil}
\email{jardim@ime.unicamp.br}

\author{Valeriano Lanza}
\address{IME - UFF \\
Departamento de Análise \\ Rua Professor Marcos Waldemar de Freitas Reis, s/n\\
24210-201 Niterói-RJ, Brazil}
\email{vlanza@id.uff.br}

\begin{document}

\begin{abstract}
In a previous paper, the first two named authors established an isomorphism between the moduli space of framed flags of sheaves on the projective plane and the moduli space of stable representations of a certain quiver. In the present note, we substitute one of the claims made, namely \cite[Theorem 17]{Ja4}, for a weaker claim regarding the existence of unobstructed points in the quiver moduli space. We also extend some of the results of the cited paper, concerning the maximal stability chamber within which the isomorphism mentioned holds, and the existence of a perfect obstruction theory for the quiver moduli space.
\end{abstract}

\maketitle

\section{Introduction}

In \cite{Ja4}, the first two named authors gave a detailed account of the ADHM construction of the moduli space of framed flags of sheaves on the projective plane, which are triples $(E,F,\varphi)$ consisting of a torsion-free sheaf $F$ on $\p2$, a framing $\varphi$ of $F$ at a line $\ell_\infty$ and a subsheaf $E\subset F$ such that the quotient $F/E$ is supported away from the framing line $\ell_\infty$. Setting $r:=\rk(E) = \rk(F)$, $n := c_2(F)$, and $l:=h^0(F/E)$, we denote by $\calf(r,n,l)$ the moduli space of such triples with these invariants fixed.

The main result of \cite{Ja4} was to show that $\calf(r,n,l)$ is isomorphic to the moduli space of representations of the \emph{enhanced ADHM quiver}. The latter is the quiver $\mathcal{Q}$ given by
\begin{equation}\label{eq:quiverADHMaument-est}
	\begin{tikzpicture}[->,>=stealth',shorten >=1pt,auto,node distance=2.5cm,
	semithick]
	\tikzstyle{every state}=[fill=white,draw=none,text=black]
	
	\node[state]    (A)                {$e_1$};
	\node[state]    (B) [left of=A]    {$e_2$};
	\node[state]    (C) [right of=A]   {$e_{\infty}$};
	
	\path (A) edge [loop above]          node {$\alpha$}     (A)
	edge [loop below]          node {$\beta$}      (A)
	edge [out=20,in=160]       node {$\eta$}       (C)
	(B) edge [loop above]          node {$\alpha'$}    (B)
	edge [loop below]          node {$\beta'$}     (B)
	edge [out=0,in=180]       node {$\phi$}       (A)
	(C) edge [out=-160,in=-20]     node {$\xi$}        (A);
	\end{tikzpicture}
\end{equation}
with the relations
\begin{equation*}
	\begin{array}{ccccc}
	\label{eq:quiverADHMenh}
	\alpha\beta - \beta\alpha + \xi\eta, & \alpha\phi -\phi\alpha', & \beta\phi-\phi\beta', &\eta\phi, & \alpha'\beta'-\beta'\alpha'\,.
	\end{array}
\end{equation*}

A \textit{representation of the enhanced ADHM quiver of type $(r,c,c')$} is a collection $X=(W,V,V',A,B,I,J,A',B',F)$, where $W$, $V$, $V'$ are complex vector spaces of dimension $r$, $c$ and $c'$, respectively, and $A,$ $B\in End(V)$, $I\in Hom(W,V)$, $J\in Hom(V,W)$, $A'$, $B'\in End(V')$ and $F\in Hom(V',V)$ satisfy the \textit{enhanced ADHM equations}:
\begin{itemize}
 \item[\bf (R1)] $[A,B]+IJ=0$;
 \item[\bf (R2)] $AF-FA'=0$; 
 \item[\bf (R3)] $BF-FB'=0$;
 \item[\bf (R4)] $JF=0$;
 \item[\bf (R5)] $[A',B']=0$.
\end{itemize}

A representation $X = (W,V,V',A,B,I,J,A',B',F)$ can be illustrated as the diagram below:

\begin{equation*}
	\begin{tikzpicture}[->,>=stealth',shorten >=1pt,auto,node distance=2.5cm,
	semithick]
	\tikzstyle{every state}=[fill=white,draw=none,text=black]
	
	\node[state]    (A)                {$V$};
	\node[state]    (B) [left of=A]    {$V'$};
	\node[state]    (C) [right of=A]   {$W$};
	
	\path (A) edge [loop above]     node {$A$}     (A)
	edge [loop below]     node {$B$}     (A)
	edge [out=20,in=160]      node {$J$}     (C)
	(B) edge [loop above]     node {$A'$}    (B)
	edge [loop below]     node {$B'$}    (B)
	edge [out=0, in=180]      node {$F$}     (A)
	(C) edge [out=-160,in=-20]      node {$I$}     (A);
	\end{tikzpicture}
\end{equation*}

Let $\varphi: W \longrightarrow \mathbb{C}^r$ be an isomorphism. Then, if $X$ is a representation of the enhanced ADHM quiver, $(X,\varphi)$ is called a \textit{framed representation} of the enhanced ADHM quiver\label{framed-representation}. Two framed representations $(X,\varphi)$ and $(\widetilde{X},\widetilde{\varphi})$ are said to be isomorphic if there exists an isomorphism
\begin{equation*}
(\xi_1,\xi_2,\xi_{\infty}): X \longrightarrow \widetilde{X},
\end{equation*}
such that $\widetilde{\varphi}\xi_{\infty}=\varphi$.

The group $\GL_{c}(\mathbb{C})\times \GL_{c'}(\mathbb{C})$ acts on the space of these ADHM data:
\[ (A,B,A',B',F,I,J)\mapsto (h Ah^{-1},h Bh^{-1},h' A'h'^{-1},h' B'h'^{-1},h Fh'^{-1},h I,Jh^{-1})\,, \]
for $h\in\GL_{c}(\mathbb{C})$, $h'\in\GL_{c'}(\mathbb{C})$. 

The moduli space $\mathcal{N}(r,c,c')$ of representations of $\mathcal{Q}$ with fixed dimension vector $(r,c,c')$ is the GIT quotient of the space of representations $X$ satisfying the relations \textbf{(R1)}-\textbf{(R5)} under the action of $\GL_{c}(\mathbb{C})\times \GL_{c'}(\mathbb{C})$; note that the space $\mathcal{N}(r,c,c')$ depends on the choice of GIT stability parameters.  

We observe that the quiver in display \eqref{eq:quiverADHMaument-est} is actually a simplified version of the enhanced ADHM quiver used in \cite{Ja4}. Although we are correcting also parts of the original paper prior to this simplification, which occurs at \cite[p. 143]{Ja4}, nothing concerning the additional map $G$, introduced at \cite[p. 141]{Ja4}, will be modified in substance. To make reading more fluent, we have therefore decided to omit the map $G$ entirely in this work.

One of the main results of \cite{Ja4} is the existence of an isomorphism of schemes $\mathfrak{b}:\mathcal{N}(r,n+l,l)\to\mathcal{F}(r,n,l)$ for every $r,n,l\ge1$, see \cite[Theorem 18]{Ja4}. 

In the present note, we retract the claim made in \cite[Theorem 17]{Ja4} that $\mathcal{N}(r,c,1)$ is non singular for every $r,c\ge 1$; the proof contains a gap that we could not circumvent; in fact, a similar issue is present in \cite[\S 3.3]{enhADHM}. Our main goal here is to argue that $\mathcal{N}(r,c,c')$ has a perfect obstruction theory for every $r,c,c'\ge1$ and that $\mathcal{N}(r,c,1)$ is generically unobstructed, meaning that it contains an open subset of points with vanishing obstruction space.

Our motivation comes from recent publications by Bonelli, Fasola and Tanzini \cite{BFT2,BFT1}, where the authors consider, in connection with a particular class of surface defects in four-dimensional supersymmetric gauge theories, representations of a quiver that generalizes the one in display \eqref{eq:quiverADHMaument-est} and is related to longer flags of sheaves on $\p2$. Since these authors quote the results of \cite{Ja4}, we felt that it was necessary to set the record straight regarding the flaws contained in \cite{Ja4}. 
It is important to emphasize, though, that Bonelli, Fasola and Tanzini mentioned in their papers only the correct results in \cite{Ja4}, avoiding the incorrect statements.

This note is organized as follows. Section \ref{sec-stability} is dedicated to generalizing \cite[Lemma 4]{Ja4} and describe the chamber in the space of stability conditions within which the isomorphism $\mathfrak{b}:\mathcal{N}(r,n + l,l) \to \mathcal{F}(r,n,l)$ holds, see Figure 1 below. We then bridge a gap in the proof of \cite[Lemma 13]{Ja4} in Section \ref{cor-lemma13}.
In Section \ref{sec:obst}, we show that $\mathcal{N}(r,c,c')$ has a perfect obstruction theory and has expected dimension equal to $2rc-rc'$ when $c'>1$, and equal to $2rc-r+1$ when $c'=1$. Finally, we provide examples of unobstructed points in $\mathcal{N}(r,c,1)$ in Section \ref{sec:sm}.

The smoothness of the moduli space $\calf(r,n,1)$ remains an important open problem; unfortunately, the ADHM construction does not seem to be the appropriate tool to tackle it. While one can check that $\calf(r,n,l)$ can indeed be singular when $l\ge2$ (see comment in the first two paragraphs of Section \ref{sec:sm}), it is still worth asking whether $\calf(r,n,l)$ is irreducible or connected for arbitrary values of $r$, $n$ and $l$; we do hope that the ADHM construction here presented will be useful to approach these questions.

\subsection*{Acknowledgments}
MJ is supported by the CNPQ grant number 302889/2018-3 and the FAPESP Thematic Project 2018/21391-1. VL was initially supported by the FAPESP post-doctoral grant number 2015/077664. The authors also acknowledge the financial support from Coordenação de Aperfeiçoamento de Pessoal de Nível Superior - Brasil (CAPES) - Finance Code 001.

\bigskip





\section{Maximal stabilty chamber}\label{sec-stability}

In the proof of \cite[Lemma 4]{Ja4} it is claimed that a certain collection of data, namely $\widetilde{X}=(W,S,V',A|_S,B|_S,I,J|_S,0,0,0)$,
is a subrepresentation of the given representation $X$. This is not true unless $F(V')\subseteq S$. We present a refined version of \cite[Lemma 4]{Ja4}, Lemma \ref{lem4} below, which enlarges the open cone of $\mathbb{R}^2$ where we pick the stability parameters (see Figure \ref{chamber}). This new open cone is sharp, i.e., it is a chamber in the space of stability parameters (cf. Lemma \ref{lemma:chamber}).

We also noticed a typo in \cite[Definition 3]{Ja4} which might mislead the reader. We shall restate therefore that definition too, fixing the typo (Definition \ref{defn:stable} below).

\bigskip


\begin{dfn}\label{defn:stable}
	\index{Representation of the enhanced ADHM quiver!$\Theta-$stable}\index{Representation of the enhanced ADHM quiver!$\Theta-$semistable}
	Let $\Theta=(\theta,\theta',\theta_{\infty})\in \mathbb{R}^3$ be a triple satisfying the relation
	\begin{equation}
	\label{eq:stability-parameter}
	c\theta + c'\theta' + r\theta_{\infty} = 0.
	\end{equation}
	A representation $X$ of numerical type $(r,c,c')$ is called \emph{$\Theta-$stable} if:
	\begin{itemize}
		\item[$(i)$] any subrepresentation $0\neq\widetilde{X}\subset X$ of numerical type $(0,\widetilde{c},\widetilde{c'})$ satisfies
		\begin{equation}
		\label{eq:thetastab01}
		\theta\widetilde{c}+\theta'\widetilde{c'}< 0;
		\end{equation}
		\item[$(ii)$] any subrepresentation $\widetilde{X}\subsetneq X$ of numerical type $(r,\widetilde{c},\widetilde{c'})$ satisfies
		\begin{equation}
		\label{eq:thetastab02}
		\theta_{\infty}r+\theta\widetilde{c}+\theta'\widetilde{c'}< 0.
		\end{equation}
	\end{itemize}
	A representation $X$ of numerical type $(r,c,c')$ is called \emph{$\Theta-$semistable} if:
	\begin{itemize}
		\item[$(iii)$] any subrepresentation $0\neq\widetilde{X}\subset X$ of numerical type $(0,\widetilde{c},\widetilde{c'})$ satisfies
		\begin{equation}
		\label{eq:thetastab03}
		\theta\widetilde{c}+\theta'\widetilde{c'}\leq 0;
		\end{equation}
		\item[$(iv)$] any subrepresentation $\widetilde{X}\subsetneq X$ of numerical type $(r,\widetilde{c},\widetilde{c'})$ satisfies
		\begin{equation}
		\label{eq:thetastab04}
		\theta_{\infty}r+\theta\widetilde{c}+\theta'\widetilde{c'}\leq 0.
		\end{equation}
	\end{itemize}\end{dfn}

\begin{lem}\label{lem4}
		Fix a triple $(r,c,c')\in(\mathbb{Z}_{>0})^3$. Suppose $\theta'>0$ and $\theta+\theta'<0$. Let $X=(W,V,V',A,B,I,J,A',B',F)$ be a representation of numerical type $(r,c,c')$. Then the following are equivalent:\begin{itemize}
			\item[$($i$)$] $X$ is $\Theta-$stable;
			\item[$($ii$)$] $X$ is $\Theta-$semistable;
			\item[$($iii$)$] $X$ satisfies the following conditions:
			\begin{enumerate}
				\item[$(S.1)$] $F\in Hom(V',V)$ is injective;
				\item[$(S.2)$] The ADHM data $\mathcal{A}=(W,V,A,B,I,J)$ is stable, i.e., there is no proper subspace $0\subset S\subsetneq V$ preserved by $A$, $B$ and containing the image of $I$.
			\end{enumerate}
		\end{itemize}
	\end{lem}
\begin{proof}
If $X$ is $\Theta-$stable, then $X$ is clearly $\Theta-$semistable.\\ We assume $X$ is $\Theta$-semistable and we want to prove condition (S.1). We notice that
		\begin{align}
		A'(\ker(F))&\subset \ker(F) \nonumber\\
		B'(\ker(F))&\subset \ker(F)\,. \nonumber 
		\end{align}
		In fact, let $v\in\ker(F)$. Then it follows from the enhanced ADHM equations that
		\begin{align}
		0 & = (BF-FB')v \nonumber\\
		& \Rightarrow F(B'v) = 0\mbox{, for all } v\in\ker(F) \nonumber \\
		& \Rightarrow B'(v)\in\ker(F) \mbox{, for all }v\in\ker(F) \nonumber\\
		& \Rightarrow B'(\ker(F))\subset \ker(F)\,. \nonumber
		\end{align}
		Analogously, the same can be proved for the endomorphism $A'$.\\
		\indent Then, $\widetilde{X} = (0,0,\ker(F),0,0,0,0,A'|_{\ker(F)},B'|_{\ker(F)},F|_{\ker(F)})$ is a subrepresentation of $X$ of numerical type $(0,0,\dim(\ker(F)))$, so that
		\begin{align*}
		\theta\widetilde{c}+\theta'\widetilde{c'} = \theta'\cdot\dim(\ker(F)) \leq 0\,.
		\end{align*}
		As $\theta'>0$, this implies
		$$
		\dim(\ker(F))=0\,,
		$$
		i.e., $F$ is injective.\\
		\indent Now we assume again that $X$ is $\Theta-$semistable (in particular, $F$ is injective) and we want to prove condition $(S.2)$. Let $S \subseteq V$ such that
		\begin{equation*}
		A(S), B(S), \im(I)\subseteq S\,,
		\end{equation*}
		so that $\widetilde{X}=(W,S,F^{-1}(S))$ is a subrepresentation of numerical type $(r,\widetilde{c},\widetilde{c}')$, where
		\[
		  \widetilde{c}'=\dim(\im(F)\cap S)=c'+\widetilde{c}-\dim(\im(F)+S)\,.
		\]
		One has
		\[
		 \theta\widetilde{c}+\theta'\widetilde{c}'=\theta\widetilde{c}+\theta'(c'+\widetilde{c}-\dim(\im(F)+S))\leq \theta c+\theta' c'\,,
		\]
				which implies
		\[
		 (\theta+\theta')\,\widetilde{c}\leq\theta c+\theta'\dim(\im(F)+S)\leq (\theta+\theta')\,c\,.
		\]
		Hence, as $\theta+\theta'<0$, one gets
		\[
		 \widetilde{c}\geq c\,,
		\]
		and so $S=V$, as wanted.
		
		To prove (iii) $\Rightarrow$ (i), let $X$ be a representation satisfying conditions (S.1) and (S.2). We want to prove that $X$ is $\Theta$-stable. We first consider nontrivial subrepresentations of the form $(0,S,S')$ and numerical type $(0,\widetilde{c},\widetilde{c}')$, with $S\subseteq \ker (J)$. We notice immediately that the injectivity of $F$ implies $\widetilde{c}\geq 1$ (if $S=0$, also $S'=0$ and $X$ would be trivial), and $\widetilde{c}'\leq\widetilde{c}$. Hence, one has
		\[
		 \theta\widetilde{c}+\theta'\widetilde{c}'\leq(\theta+\theta')\widetilde{c}< 0\,,
		\]
		as wanted.
		
		Finally, we consider proper subrepresentations of the form $(W,S,S')$ and numerical type $(r,\widetilde{c},\widetilde{c}')$, with $S\supseteq\im(I)$. By condition (S.2), $\widetilde{c}=c$, so that $\widetilde{c}'<c'$ (in order for $X$ to be proper). One has
		\[
		  \theta\widetilde{c}+\theta'\widetilde{c}'=\theta c+\theta'\widetilde{c}'<\theta c+\theta'c'\,,
		\]
		as wanted.
\end{proof}

In the plane $\mathbb{R}^2_{\theta,\theta'}$ we focused so far on the open cone
\[
 \Delta=\{\theta'>0\,,\ \theta<-\theta'\}\, ,
\]
which is pictured in Figure \ref{chamber}.

\begin{figure}
    \centering
    \includegraphics{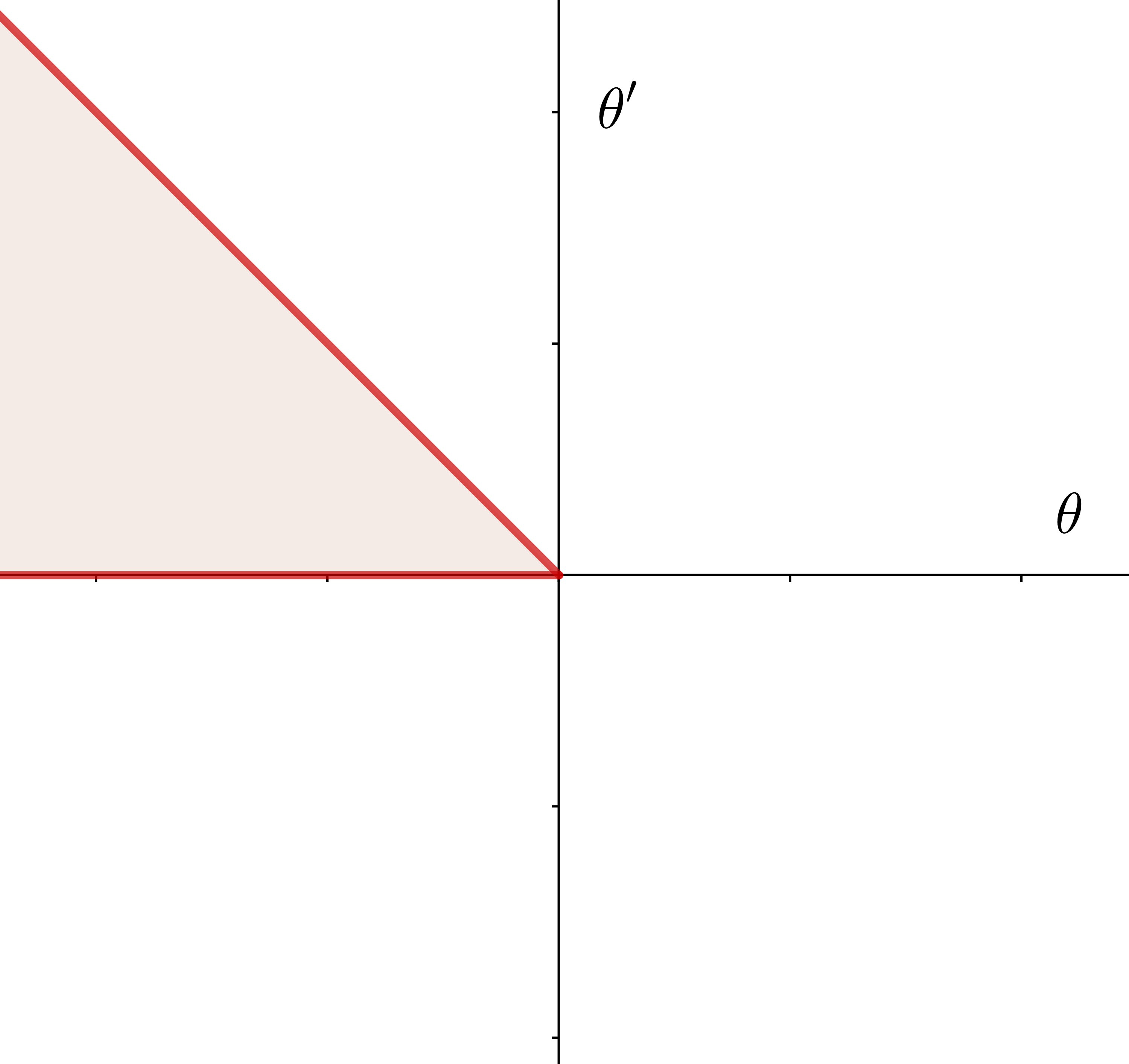}
    \caption{Stability chamber in the $(\theta,\theta')$-plane.}
    \label{chamber}
\end{figure}

One can wonder if this is indeed a chamber in the space of parameters. To understand it, we need to investigate separately what happens on the open rays
\[
  \rho^-=\{\theta'=0\,,\ \theta<0\}\quad\text{and}\quad\rho^+=\{\theta'>0\,,\ \theta=-\theta'\}\,.
\]
Note that if $\Theta_1,\Theta_2\in\rho^-$, then an object is $\Theta_1$-semistable if and only if it is $\Theta_2$-semistable; a similar fact holds along $\rho^+$. With this in mind, we shall denote generically by $\Theta^-$ and $\Theta^+$ any stability parameter belonging to $\rho^-$ and $\rho^+$, respectively.

\begin{lem}\label{lemma:chamber}
There exist strictly $\Theta^-$- and $\Theta^+$-semistable representations of numerical type $(r,c,c')$.
\end{lem}

\begin{proof}
By inspecting Definition \ref{defn:stable} one can easily see that a representation $X$ is $\Theta^-$-stable if it satisfies the conditions:
\begin{itemize}
    \item[(R.1s)] for any subrepresentation of numerical type $(0,\widetilde{c},\widetilde{c}')$, one has $\widetilde{c}>0$ (unless $\widetilde{c}=\widetilde{c}'=0$);
   \item[(R.2s)] for any subrepresentation of numerical type  $(r,\widetilde{c},\widetilde{c}')$, one has $\widetilde{c}=c$ and $\widetilde{c}'=c'$.  
\end{itemize}

A representation $X$ is $\Theta^-$-semistable if it satisfies instead the sole condition:
\begin{itemize}
    \item[(R.2)] for any subrepresentation of numerical type $(r,\widetilde{c},\widetilde{c}')$, one has $\widetilde{c}=c$.  
\end{itemize}

We consider the representation $X^-=(W,V,V',\bar{A},\bar{A},\bar{I},0,0,0,0)$, where, after fixing a basis $e_1,\dots,e_c$ of $V$, the linear morphisms $\bar{A}$ and $\bar{I}$ are defined as
\begin{equation}\label{eq:es}
 \bar{A}=\left(
\begin{array}{c|c}
0 & 0\\
\mathbf{1}_{c-1} & 0
\end{array}
\right)\quad\text{and}\quad\bar{I}=
\left(
\begin{array}{c|c}
1 & 0\\
0 & \\
\vdots&\mathbf{0}_{(c-1)\times(r-1)}\\
0&
\end{array}
\right)\,.
\end{equation}
We claim that $X^-$ is strictly $\Theta^-$-semistable, so that $\rho^-$ is indeed a wall. To verify first that $X^-$ is $\Theta^-$-semistable we have to consider subrepresentations of the form $\widetilde{X}=(W,S,S')$ with $S\supseteq\im(\bar{I})=\langle e_1\rangle$, and prove that $S=V$. But this is readily checked once we notice that
\[
 e_i=\bar{A}^{i-1}e_1\in S\quad\text{for $i=1,\dots,c$.}
\]
However, $X^-$ is not $\Theta^-$-stable, as for instance the subrepresentation $\widetilde{X}=(0,0,V')$ violates condition (R.1s).

By inspecting again \cite[Definition 3]{Ja4} one can easily see that a representation $X$ is $\Theta^+$-stable if it satisfies the conditions:
\begin{itemize}
    \item[(T.1s)] for any subrepresentation of numerical type $(0,\widetilde{c},\widetilde{c}')$, one has $\widetilde{c}'<\widetilde{c}$ (unless $\widetilde{c}=\widetilde{c}'=0$);
    \item[(T.2s)] for any subrepresentation of numerical type $(r,\widetilde{c},\widetilde{c}')$, one has $\widetilde{c}'<\widetilde{c}+c'-c$ (unless $\widetilde{c}=c$ and $\widetilde{c}'=c'$).  
\end{itemize}
A representation $X$ is $\Theta^+$-semistable if it satisfies instead the conditions:
\begin{itemize}
    \item[(T.1)] for any subrepresentation of numerical type $(0,\widetilde{c},\widetilde{c}')$, one has $\widetilde{c}'\leq \widetilde{c}$\,;  
    \item[(T.2)] for any subrepresentation of numerical type $(r,\widetilde{c},\widetilde{c}')$, one has $\widetilde{c}'\leq\widetilde{c}+c'-c$.  
\end{itemize}

We consider the representation $X^+=(W,V,V',\bar{A},\bar{A},\bar{I},0,\bar{A}',\bar{A}',\bar{F})$, where, after fixing bases $e_1,\dots,e_c$ and $e'_1,\dots,e_{c'}$ of $V$ and $V'$, respectively, the linear morphisms $\bar{A}$ and $\bar{I}$ are defined as in eq.~\eqref{eq:es}, while
\[
 \bar{A}'=\left(
\begin{array}{c|c}
0 & 0\\
\mathbf{1}_{c'-1} & 0
\end{array}
\right)\quad\text{and}\quad
 \bar{F}=
 \begin{pmatrix}
   \mathbf{0}_{(c-c')\times c'}\\\mathbf{1}_{c'}
 \end{pmatrix}\,.
\]
We claim that $X^+$ is strictly $\Theta^+$-semistable, so that for such invariants $\rho^+$ is indeed a wall. Condition (T.1) holds true as $\bar{F}$ is injective; as for condition (T.2) one can arguing as already done for $X^-$ (notice that condition (R.2) is stronger than condition (T.2)). However, $X^+$ is not $\Theta^+$-stable, as for instance the subrepresentation $(0,\langle e_{c-c'+1},\dots,e_c\rangle,V')$ violates condition (T.1s).
\end{proof}

\begin{obs}
 It is immediate to check that the possible presence of the map $G$ (cf. Introduction) does not affect any of the arguments above, so that the stability chamber in Figure \ref{chamber} is indeed a stability chamber for the (original) enhanced ADHM quiver.
\end{obs}
\bigskip


\section{Fibration over the moduli space of torsion-free sheaves on $\p2$}\label{cor-lemma13}

A different issue concerns the proof of \cite[Lemma 13]{Ja4}: although the result per se stands true, the matrix-theoretical proof provided there presents a gap which we were not able to bypass. As a consequence, the validity of both \cite[Lemma 14]{Ja4} (connectedness of $\mathcal{N}(r,c,1)$) and \cite[Corollary 22]{Ja4} (irreducibility of $\mathcal{F}(r,c,1)$) is compromised.

With any point \[[(W,V,V',A,B,I,J,A',B',F)]\in\mathcal{N}(r,c,c')\,,\] one can associate a representation $(W,V'',A'',B'',I'',J'')$ of the usual ADHM quiver as follows: we set $V''=V/im(F)$, and we denote by $\pi''$ the canonical projection $V\longrightarrow V''$; after that, we let $I''$ be the composition $\pi''\circ I$; finally, we notice that, thanks to the relations \textbf{(R2)}-\textbf{(R4)}, the linear maps $A,B$ and $J$ induce maps $A'',B''\in End(V'')$, and $J''\in Hom(V'',W)$, respectively. One can check that this is indeed a representation of the ADHM quiver, i.e., it satisfies the ADHM equation, and that it is actually stable (see \cite[page 148]{Ja4} for details). Therefore, one has a well-defined morphism
\begin{equation}\label{eq:q}
 \mathfrak{q}\colon \mathcal{N}(r,c,c') \longrightarrow\mathcal{M}(r,c-c')\,,
\end{equation}
where $\mathcal{M}(r,c-c')$ is the moduli space of the stable representations of the ADHM quiver of numerical type $(r,c-c')$.

\begin{lem}$($\cite[Lemma 13]{Ja4}$)$
The morphism $\mathfrak{q}$ in eq.~\eqref{eq:q} is surjective. 
\end{lem}
\begin{proof}[New proof]
It follows from \cite[Lemma 2]{Ja4} and \cite[Lemma 18]{Ja4} that there is a commutative diagram of schemes
$$ \xymatrix{
\mathcal{N}(r,n+l,l) \ar[rr]^{\mathfrak{b}}_{\sim} \ar[dr]^{\mathfrak{q}} & & \mathcal{F}(r,n,l) \ar[ld]^{\mathfrak{p}} \\
& \mathcal{M}(r,n) &
} $$
where $\mathfrak{p}$ is surjective and $\mathfrak{b}$ bijective. Let $\mathfrak{p}'$ be a right inverse of $\mathfrak{p}$. It is easy to check that \[\mathfrak{q'}:=\mathfrak{b}^{-1}\circ \mathfrak{p}'\] is a right inverse of $\mathfrak{q}$. In other words, $\mathfrak{q}$ is surjective.
\end{proof}

The original proof of \cite[Lemma 13]{Ja4} was actually saying more; namely, it was providing an explicit procedure to concoct an element of the fiber of $\mathfrak{q}$ over an arbitrary point $[(W,V'',A'',B'',I'',J'')]\in\mathcal{M}(r,c-c')$. The idea was to consider the vector space $V=\mathbb{C}^{c'}\oplus V''$, and, after choosing bases for $V''$ and $W$, search for linear maps of the form
\begin{equation}\label{eq:surj}
\begin{gathered}
		A=\begin{pmatrix}
		 	A' & \widetilde{A} \\
		0  & A''
		\end{pmatrix},\quad B=\begin{pmatrix}
		B' & \widetilde{B}\\
		0  & B''
	\end{pmatrix},\\
	    I =\begin{pmatrix}
		\widetilde{I}\\
		I''
		\end{pmatrix},\quad J =\begin{pmatrix}
		0 & J''
	\end{pmatrix},\quad
	F = \begin{pmatrix}
         \mathbf{1}_{c'}\\
         0
        \end{pmatrix}
\end{gathered}
\end{equation}
For such a collection of data, the set of relations \textbf{(R1)}-\textbf{(R5)} turns out to be equivalent to the system
\begin{equation} \label{eq:lem-enhADHMeqobt02}
	\left\{\begin{array}{l}
	[A',B'] = 0 \\
	A'\widetilde{B}+ \widetilde{A}B''-B'\widetilde{A}-\widetilde{B}A''+\widetilde{I}J'' = 0\end{array}\right. .
\end{equation}
As for stability, it was incorrectly stated in the proof of \cite[Lemma 13]{Ja4} (see \emph{op.cit.}, page 149, after eq.~(32)) that the linear data in eq.~\eqref{eq:surj} are stable if and only if:
\begin{itemize}\label{decomp_stable}
\item[$(i)$] at least one of the maps $\widetilde{A}$, $\widetilde{B}$ and $\widetilde{I}$ is nontrivial;
\item[$(ii)$] there is no proper subspace $S'\subsetneq V'$ such that
\begin{align*}\label{prop:minilemmainside}
\widetilde{A}(V''),\quad \widetilde{B}(V''),\quad \widetilde{I}(W)\subset S'\quad \mbox{and}\quad A'(S'), B'(S')\subset S'.
\end{align*}
\end{itemize}
Actually, these two conditions happen to be only necessary, as it is possible to build examples of unstable representations that satisfy them.

Whilst we do believe that stable solutions for the system in eq.~\eqref{eq:lem-enhADHMeqobt02} can be found explicitly, our attempts showed a high computational complexity as soon as the invariants grow; in particular, we are not able to preserve the structural simplicity of the solutions in \cite{Ja4}, which was pivotal to the subsequent proofs of \cite[Lemma 14]{Ja4} and \cite[Corollary 22]{Ja4}.


\section{Obstruction theory} \label{sec:obst}

In \cite[\S 5]{Ja4} it is claimed that the moduli spaces $\mathcal{N}(r,c,1)$ are smooth. The proof, however, contains a flaw which we could not circumvented (we also suggest the reader to compare also the published and current arXiv versions of \cite[\S3.3]{enhADHM}).

We start by recovering the main steps (with few technical changes) to get the more general result expressed by \cite[Theorem 15]{Ja4}, i.e., that the moduli spaces $\mathcal{N}(r,c,c')$ have a perfect obstruction theory.

Given the representation $X=(W,V,V',A,B,A',B',F,I,J)$ we associate with the class $[X]$ the deformation complex $\mathcal{C}(X)$:
\begin{equation}\label{eq:cotan}
 \begin{array}{c}\End(V)\\\oplus\\\End(V')\end{array}\stackrel{d_0}{\longrightarrow}
 \begin{array}{c}\End(V)^{\oplus 2}\\\oplus\\\Hom(W,V)\\\oplus\\\Hom(V,W)\\\oplus\\\End(V')^{\oplus 2}\\\oplus\\\Hom(V',V)\end{array}
 \stackrel{d_1}{\longrightarrow}
 \begin{array}{c}\End(V)\\\oplus\\\Hom(V',V)^{\oplus 2}\\\oplus\\\Hom(V',W)\\\oplus\\\End(V')\end{array}
 \stackrel{d_2}{\longrightarrow}\ \Hom(V',V)\,, 
\end{equation}
where
\begin{gather*}
 d_0\begin{pmatrix}
     h\\
     h'
    \end{pmatrix}=
    \begin{pmatrix}
     [h,A]\\
     [h,B]\\
     hI\\
     -Jh\\
     [h',A']\\
     [h',B']\\
     hF-Fh'
    \end{pmatrix}\,,\qquad
d_1\begin{pmatrix}
     a\\
     b\\
     i\\
     j\\
     a'\\
     b'\\
     f
   \end{pmatrix}=
    \begin{pmatrix}
     [a,B]+[A,b]+Ij+iJ\\
     Af+aF-Fa'-fA'\\
     Bf+bF-Fb'-fB'\\
     jF+Jf\\
     [a',B']+[A',b']
    \end{pmatrix}\,,\\
    d_2\begin{pmatrix}
        c_1\\
        c_2\\
        c_3\\
        c_4\\
        c_5
       \end{pmatrix}=c_1F+Bc_2-c_2B'-Ac_3+c_3A'-Ic_4-Fc_5\,.
\end{gather*}

We introduce three auxiliary complexes (slightly different than those in \cite{Ja4}): 
\[
 \mathcal{C}(X')\colon\qquad\End(V)\ \stackrel{d_0}{\longrightarrow}
 \begin{array}{c}\End(V)^{\oplus 2}\\\oplus\\\Hom(W,V)\\\oplus\\\Hom(V,W)\end{array}
 \stackrel{d_1}{\longrightarrow}\ \End(V)\,,
\]
with
\begin{gather*}
 d_0(h)=
    \begin{pmatrix}
     [h,A]\\
     [h,B]\\
     hI\\
     -Jh
    \end{pmatrix}\,,\qquad
d_1\begin{pmatrix}
     a\\
     b\\
     i\\
     j
   \end{pmatrix}=
    [a,B]+[A,b]+Ij+iJ\,;
\end{gather*}\medskip
\[\
 \mathcal{C}(X'')\colon\qquad\End(V')\ \stackrel{d_0}{\longrightarrow}\ \End(V')^{\oplus 2}\ \stackrel{d_1}{\longrightarrow}\ \End(V')\,,
\]
with
\begin{gather}\label{eq:R''}
 d_0(h')=\begin{pmatrix}
           [h',A']\\
           [h',B']
          \end{pmatrix}\,,\qquad
d_1\begin{pmatrix}
    a'\\
    b'
   \end{pmatrix}=[a',B']+[A',b']\,;
\end{gather}\medskip
\[
 \mathcal{C}(X',X'')\colon\qquad\Hom(V',V)\ \stackrel{d_0}{\longrightarrow}
 \begin{array}{c}\Hom(V',V)^{\oplus 2}\\\oplus\\\Hom(V',W)\end{array}
 \stackrel{d_1}{\longrightarrow}\ \Hom(V',V)\,,
\]
with
\begin{gather*}
 d_0(f)=
    \begin{pmatrix}
     -Af+fA'\\
     -Bf+fB'\\
     -Jf
    \end{pmatrix}\,,\qquad
d_1\begin{pmatrix}
     c_2\\
     c_3\\
     c_4
   \end{pmatrix}=
    -Bc_2+c_2B'+Ac_3-c_3A'+Ic_4\,.
\end{gather*}
\begin{obs}
Notice that $\mathcal{C}(X')$ and $\mathcal{C}(X'')$ are simply the deformation complexes associated with the stable representations of the usual ADHM quiver $X'=(W,V,A,B,I,J)$ and $X''=(0,V',A',B')$, respectively, which explains the notation.
\end{obs}

The following lemma is readily checked.
\begin{lem}
There is a distinguished triangle
\begin{equation}\label{eq:triangle}
 \mathcal{C}(X)\longrightarrow\mathcal{C}(X')\oplus\mathcal{C}(X'')\stackrel{\rho}{\longrightarrow}\mathcal{C}(X',X'')\,,
\end{equation}
where $\rho=(\rho_0,\rho_1,\rho_2)$ with
\begin{gather*}
 \rho_0\begin{pmatrix}
     h\\
     h'
    \end{pmatrix}=-hF+Fh'\,,\qquad
\rho_1\begin{pmatrix}
     a\\
     b\\
     i\\
     j\\
     a'\\
     b'
     \end{pmatrix}=
    \begin{pmatrix}
     -aF+Fa'\\
     -bF+Fb'\\
     -jF
    \end{pmatrix}\,,\\
 \rho_2\begin{pmatrix}c_1\\c_5\end{pmatrix}=-c_1F+Fc_5\,.   
\end{gather*}
\end{lem}

 
\begin{thm}\label{prop:vanishing}
 The moduli space $\mathcal{N}(r,c,c')$ has a perfect obstruction theory, i.e., for every $[X]\in\mathcal{N}(r,c,c')$, $H^0(\mathcal{C}(X))=H^3(\mathcal{C}(X))=0$.
\end{thm}
 
\begin{proof}
 From the triangle in \eqref{eq:triangle} one gets the long exact sequence
 \begin{multline}\label{eq:long}\xymatrix{
  0\ar[r]& H^0(\mathcal{C}(X))\ar[r]& {H^0(\mathcal{C}(X'))\oplus H^0(\mathcal{C}(X''))}\ar[r]^-{H^0(\rho)}& H^0(\mathcal{C}(X',X''))
 }\\
 \xymatrix{
  {}\ar[r]& H^1(\mathcal{C}(X))\ar[r]& {H^1(\mathcal{C}(X'))\oplus H^1(\mathcal{C}(X''))}\ar[r]^-{H^1(\rho)}& H^1(\mathcal{C}(X',X''))\ar[r]& H^2(\mathcal{C}(X))
 }\\
 \xymatrix{
  {}\ar[r]& {H^2(\mathcal{C}(X'))\oplus H^2(\mathcal{C}(X''))}\ar[r]& H^2(\mathcal{C}(X',X''))\ar[r]&H^3(\mathcal{C}(X))\ar[r]&0\,.
 }
 \end{multline}
Since $X'$ is a stable representation of the usual ADHM quiver (in that case, the smoothness of the moduli space is well-known), one has \[H^0(\mathcal{C}(X'))=H^2(\mathcal{C}(X'))=0\,.\]

We claim that $H^0(\rho)\colon H^0(\mathcal{C}(X''))\to H^0(\mathcal{C}(X',X''))$ is injective. In fact, one has
\[
 H^0(\rho)(h')=0\ \Rightarrow Fh'=0\ \Rightarrow\ h'=0\,,
\]
as $F$ is injective. This implies $H^0(\mathcal{C}(X))=0$.

To conclude the proof, we show that $d_1\colon \mathcal{C}(X',X'')^1\to\mathcal{C}(X',X'')^2$ is surjective. This implies $H^2(\mathcal{C}(X',X''))=0$, and so $H^3(\mathcal{C}(X))=0$. We claim equivalently that the dual
\[
 d_1^\vee\colon \Hom(V,V')\to
 \begin{array}{c}
  \Hom(V,V')^{\oplus 2}\\
  \oplus\\
  \Hom(W,V')
 \end{array}\ ,\qquad d_1^\vee(\phi)=
 \begin{pmatrix}
  B'\phi-\phi B\\
  -A'\phi+\phi A\\
  \phi I
 \end{pmatrix}
\]
is injective. In fact, it is easy to show that, if $\phi\in\ker(d_1^\vee)$, then $\ker(\phi)$ is an $(A,B)$-invariant subspace that contains $\im(I)$. By $($S.2$)$, $\ker(\phi)=V$, that is $\phi=0$.
\end{proof}

\begin{obs}
 Directly from the deformation complex in \eqref{eq:cotan}, one sees that the moduli space $\mathcal{N}(r,c,c')$ has expected dimension $r(2c-c')$ when $c'>1$; the case $c'=1$, will be treated below.
\end{obs}


\section{Unobstructed representations} \label{sec:sm}

The moduli spaces $\mathcal{N}(r,c,c')$ are not smooth in general. For  $c=3$, $c'=2$, and arbitrary rank, specific examples can be found by explicitly choosing suitable representations, and showing that the dimension of the first cohomology groups of their deformation complexes \eqref{eq:cotan} is not constant (through a direct computation).

However, it is already known \cite{nested} that $\mathcal{N}(1,c,c')$ (which corresponds to the nested Hilbert scheme of points of $\mathbb{P}^2$) is smooth precisely when $c'=1$. One might wonder if this numerical assumption guarantees smoothness also for higher rank, but this problem is still open, at the moment; nonetheless, we show now that $\mathcal{N}(r,c,1)$ is at least \emph{generically unobstructed}, meaning that it always contains points for which the second cohomology group of the deformation complex vanishes.

We start by noticing that the condition $c'=1$ provides a lot of extra commutations (in particular, relation \textbf{(R5)} trivializes), so that the deformation complex $\mathcal{C}(X)$ in eq.~\eqref{eq:cotan} can be simplified as follows:
\begin{equation}\label{eq:compl-c'=1}
 \begin{array}{c}\End(V)\\\oplus\\\End(V')\end{array}\stackrel{d_0}{\longrightarrow}
 \begin{array}{c}\End(V)^{\oplus 2}\\\oplus\\\Hom(W,V)\\\oplus\\\Hom(V,W)\\\oplus\\\End(V')^{\oplus 2}\\\oplus\\\Hom(V',V)\end{array}
 \stackrel{d_1}{\longrightarrow}
 \begin{array}{c}\End(V)\\\oplus\\\Hom(V',V)^{\oplus 2}\\\oplus\\\Hom(V',W)\end{array}
 \stackrel{d_2}{\longrightarrow}\ \Hom(V',V)\,,
\end{equation}
where
\begin{gather*}
 d_0\begin{pmatrix}
     h\\
     h'
    \end{pmatrix}=
    \begin{pmatrix}
     [h,A]\\
     [h,B]\\
     hI\\
     -Jh\\
     0\\
     0\\
     hF-Fh'
    \end{pmatrix}\,,\qquad
d_1\begin{pmatrix}
     a\\
     b\\
     i\\
     j\\
     a'\\
     b'\\
     f
   \end{pmatrix}=
    \begin{pmatrix}
     [a,B]+[A,b]+Ij+iJ\\
     Af+aF-Fa'-fA'\\
     Bf+bF-Fb'-fB'\\
     jF+Jf
    \end{pmatrix}\,,\\
    d_2\begin{pmatrix}
        c_1\\
        c_2\\
        c_3\\
        c_4
       \end{pmatrix}=c_1F+Bc_2-c_2B'-Ac_3+c_3A'-Ic_4\,.
\end{gather*}
The complexes $\mathcal{C}(X')$ and $\mathcal{C}(X',X'')$ do not change, while $\mathcal{C}(X'')$ becomes
\begin{equation}\label{eq:CX''}
 \mathcal{C}(X'')\colon\qquad\End(V')\ \stackrel{0}{\longrightarrow}\ \End(V')^{\oplus 2}\,,
\end{equation}
and the map $\rho$ changes only in the component $\rho_2(c_1)=-c_1F$.

The simplification in the deformation complex does not affect Theorem \ref{prop:vanishing} (which was in fact stated in full generality); however, note that the expected dimension computed through eq.~\eqref{eq:compl-c'=1} differs from the general formula, being equal to $2rc-r+1$ (instead of $2rc-r$).

\begin{pps}
For any choice of $\lambda_1,\dots,\lambda_c\in\mathbb{C}\setminus\{0\}$ with $\lambda_i\neq\lambda_j$, for $i\neq j$, consider the representation $X=(\mathbb{C}^r,\mathbb{C}^c,\mathbb{C},A,B,I,0,A',B',F)$, where
\begin{equation}\label{X-suave}
    A= B = \begin{pmatrix}
            \lambda_1 & \cdots & 0 \\
            \vdots & \ddots & \vdots \\
            0 & \cdots & \lambda_c
    \end{pmatrix}\,,\quad
    I = \begin{pmatrix}
1 & 0 & \cdots & 0 \\
\vdots & \vdots & \ddots & \vdots \\
1 & 0 & \cdots & 0
    \end{pmatrix}\,,\quad
    F = \begin{pmatrix}
    1\\
    0\\
    \vdots\\
    0
    \end{pmatrix}\,,\quad
    A'=B'= \begin{pmatrix}\lambda_1\end{pmatrix}\,.
\end{equation}
\begin{itemize}
    \item[(i)] $X$ is stable;
    \item[(ii)] $[X]\in \mathcal{N}(r,c,1)$ is unobstructed.
\end{itemize}
\end{pps}
\begin{proof}
The map $F$ is clearly injective. To prove condition $($S.2$)$, assume there is a subspace $S\subseteq \mathbb{C}^c$ such that $A(S), B(S), \im I\subseteq S$; since $v := (1 \, \cdots \, 1)^T\in \im I\subset S$, also
 \begin{align*}
     v_k:= A^k(v) = 
     \begin{pmatrix}
    \lambda_1^k\\
\vdots \\
    \lambda_{c}^k
    \end{pmatrix} \in S\qquad\text{for all $k\in\mathbb{N}$.}
 \end{align*}
  On the other hand, $\{v, v_1, \cdots, v_{c-1}\}$ is a basis of $\mathbb{C}^c$: indeed, the matrix
 \begin{align*}
     M :=   \begin{pmatrix}
            v &|& v_1 &|& \cdots &|& v_{c-1} 
            \end{pmatrix}
\end{align*}
 is a Vandermonde matrix (hence, invertible).
 We conclude $S = \mathbb{C}^c$, and $X$ is stable.
 
 To prove item (ii), we first notice that, by Proposition \ref{prop:vanishing} and eq.~\eqref{eq:CX''}, the long exact sequence in eq.~\eqref{eq:long} reduces to
\begin{multline}\label{eq:long2}
 \xymatrix{
  0\ar[r]& H^0(\mathcal{C}(X''))\ar[r]^-{H^0(\rho)}& H^0(\mathcal{C}(X',X''))\ar[r]& H^1(\mathcal{C}(X))\ar[r]&{}} \\
 \xymatrix{
  {H^1(\mathcal{C}(X'))\oplus H^1(\mathcal{C}(X''))
 }\ar[r]^-{H^1(\rho)}& H^1(\mathcal{C}(X',X''))\ar[r]& H^2(\mathcal{C}(X))\ar[r]&0\,.
 }
\end{multline}
To deduce $H^2(\mathcal{C}(X))=0$ it would be enough to show that $H^1(\rho)$ is surjective; we prove instead an even stronger result, that is, that $\rho_1$ is surjective on cocyles.

Let $(c_2 ~~ c_3 ~~ c_4)^{\textrm{T}}$ be an element of $\mathcal{C}(X',X'')^1$ such that
\begin{align}\label{cocycle01}
    d_1\begin{pmatrix}
     c_2\\
     c_3\\
     c_4
    \end{pmatrix}=-Bc_2+c_2B'+Ac_3-c_3A'+Ic_4=0\,. 
\end{align}
 If we define $F':=\begin{pmatrix}
     1&0&\cdots&0
 \end{pmatrix}$ (a left inverse for $F$), one has
\[
 \begin{pmatrix}
  c_2\\
  c_3\\
  c_4
 \end{pmatrix}=\rho_1
 \begin{pmatrix}
  -c_2F'\\
  -c_3F'\\
  i\\
  -c_4F'\\
  0\\
  0
 \end{pmatrix}\,,
\]
for any choice of $i\in\Hom(\mathbb{C}^r,\mathbb{C}^c)$ (in particular, this shows that $\rho_1$ is surjective). To end the proof it is enough to show that there is a counterimage which is actually a cocycle. We claim that this corresponds to the choice $i=0$. In fact, as $F'A=A'F'$ (and $F'B=B'F'$),
\begin{align}\label{cocycle02}
 d_1\begin{pmatrix}
  -c_2F'\\
  -c_3F'\\
  0\\
  -c_4F'\\
  0\\
  0
 \end{pmatrix}&=[-c_2F',B]+[A,-c_3F']-Ic_4F' \nonumber \\
              &=(-c_2B+Bc_2-Ac_3'+c_3A'-Ic_4)F'\,,
\end{align}
and this vanishes by eq.~\eqref{cocycle01}, as wanted.
\end{proof}

\frenchspacing

\bibliographystyle{siam}

\bibliography{bibliografia-LC}

\end{document}